\documentclass[11pt]{amsart}

\usepackage{amsmath,amssymb,amsthm}

\newtheorem{thm}{Theorem}[section]

\newtheorem{lem}[thm]{Lemma}
\newtheorem{cor}[thm]{Corollary}

\theoremstyle{definition}

\newtheorem{remark}[thm]{Remark}

\newcommand{\pr}[1]{{\mathbb P}^{#1}}
\newcommand{\skipit}[1]{{}}
\newcommand{\reg}{\operatorname{reg}}
\newcommand{\prfend}{\hbox to7pt{\hfil}
\par\vskip-\baselineskip\hbox to\hsize
{\hfil\vbox {\hrule width6pt height6pt}}\vskip\baselineskip}

\newcommand{\myarrow}[2]{\hbox to #1pt{\hfil$\to$\hfil}{\hskip-#1pt{\raise
10pt\hbox to#1pt{\hfil$\scriptscriptstyle #2$\hfil}}}}

\begin{document}

\title[Fat linear subspaces] {Inclics, Galaxies, Star Configurations and Waldschmidt constants}

\author{Giuliana Fatabbi, Brian Harbourne and Anna Lorenzini}

\address{
Dip. di Matematica e Informatica\\
Universit\`a di Perugia\\
via Vanvitelli 1\\
06123 Perugia, Italy} \email{fatabbi@dmi.unipg.it}

\address{
Department of Mathematics\\
University of Nebraska\\
Lincoln, NE 68588-0130 USA}
\email{brianharbourne@unl.edu}

\address{
Dip. di Matematica e Informatica\\
Universit\`a di Perugia\\
via Vanvitelli 1\\
06123 Perugia, Italy} \email{annalor@dmi.unipg.it}

\date{May 9, 2013}

\begin{abstract}
This paper introduces complexes of linear varieties,
called inclics (for INductively Constructible LInear ComplexeS).
As examples, we study galaxies (these are constructed starting with
a star configuration to which we add general points in a larger projective space).
By assigning an order of vanishing
(i.e., a multiplicity) to each member of the complex,
we obtain fat linear varieties (fat points if all of the linear varieties are points).
The scheme theoretic union of these fat linear varieties gives an inclic
scheme $X$. For such a scheme, we show there is an inductive procedure for
computing the Hilbert function of its defining ideal $I_X$,
regardless of the choice of multiplicities. As an application, we show how our results
allow the computation of the Hilbert functions of, for example, symbolic powers $(I_X)^{(m)}$ for arbitrary $m$
of many new examples of radical ideals $(I_X)$, and we explicitly
compute the Waldschmidt constants $\gamma(I_X)$ for galactic inclics $X$.
\end{abstract}

\subjclass[2000]{Primary:
13F20, 
14C20; 
Secondary:
13A02, 
14N05} 

\keywords{star configurations, galaxies, inclic schemes,
symbolic powers, fat points, projective space, Hilbert functions}

\thanks{The second author's work on this project was sponsored by the National Security Agency under Grant/Cooperative agreement ``Advances on Fat Points and Symbolic Powers,'' Number H98230-11-1-0139. The United States Government is authorized to reproduce and distribute reprints notwithstanding any copyright notice. The second author also
thanks the University of Perugia and his coauthors for their generous hospitality
during his visit in 2010 when this work was initiated.}

\maketitle

\section{Introduction}\label{intro}
There is a long tradition of research on ideals of
unions of linear varieties in projective spaces.
Such an ideal is the intersection of ideals generated by linear forms.
Examples include square free monomial ideals,
ideals of star configurations \cite{GHM} and ideals of finite sets of points.
Research started with the radical case (see \cite{D, DS, GO, HH, HS, L} for example)
but there is also a lot of interest in ideals of
schemes of linear varieties with assigned multiplicities, including but not limited to fat points
(see \cite{CHT, Fa, FHL, FaL, FrL, Fr, FMN, valla} for just a few examples).
The ideals in the uniform multiplicity case are symbolic power ideals;
ideals in this special case are also of interest and are receiving increasing attention
(see \cite{BH, BH2, GHM, GMS, GHV2, HaHu, M} for example), but there are
few cases where the Hilbert functions of arbitrary symbolic powers can be determined.

In this paper we introduce {\em inclic schemes}. These are schemes whose components
comprise a complex of linear varieties
called an inclic (for INductively Constructible LInear Complex).
An inclic scheme is obtained by arbitrarily assigning a multiplicity (i.e., an assigned order of vanishing)
to each component. Our main foundational result provides a recursive procedure for computing Hilbert functions of
ideals of inclic schemes (in a forthcoming paper we will study minimal free resolutions).
In certain cases this procedure can be applied to compute
Hilbert functions of arbitrary symbolic powers of radical ideals. This substantially extends the range of examples
of ideals for which this is possible. As an application, we define galactic schemes
and explicitly compute the Waldschmidt constants of certain galactic schemes built up from star configurations.
(A Waldschmidt constant is an asymptotic measure of the initial degrees
of the symbolic powers of an ideal. These have arisen in a range of previous research
(such as \cite{BH, Ch, DHST, GHV1, HaHu, M, W, W2})
and are related to work on multiplier ideals
(see \cite{EV, HaHu} and \cite[Proposition 10.1.1 and Example 10.1.3]{La}).)

\subsection{Inclics}\label{inclic}
To define inclics, let $n>0$ be an integer. We work in the projective space
$\pr {n}$ over an arbitrary field $K$ (some results will require
the characteristic to be 0). An {\em inclic}
$${\mathcal C}={\mathcal C}(n,r,s; L_0, L_1,\ldots,L_r, H_0, H_1,\ldots,H_s)$$
is a collection of linear subvarieties $L_0, L_1,\ldots,L_r, H_0, H_1,\ldots,H_s\subsetneq \pr {n}$
such that the following conditions hold:
\begin{itemize}
\item [(C1)] $H_0, H_1,\ldots,H_s$ are distinct hyperplanes;
\item [(C2)] $L_i\subseteq H_0$ for $i>0$ but $L_0\not\subseteq H_0$;
\item [(C3)] if $L_i\subseteq L_j$, then $i=j$; and
\item [(C4)] for all $i\geq 0$ and $j>0$ we have $L_i\not\subseteq H_j$.
\end{itemize}

If $s=0$ and each $L_i$ is a point, then the inclic is just a choice of $r$ points $L_i$, $0<i\leq r$, of the hyperplane $H_0$
and one point $L_0$ not on $H_0$.

\subsection{Galaxies}\label{gals}
An interesting special case of an inclic comes from what we refer to as a {\em galaxy}.
For this, we start with a star configuration. We recall \cite{GHM} that a star
configuration in $\pr n$ is defined by a set of
$u>n$ hyperplanes $A_1,\ldots,A_u\subset\pr n$ such that the
intersection of any $i$ of the hyperplanes
has dimension at most $n-i$. The star configuration of
codimension $e\leq n$ is the set $S(n,e,u)$ of the $\binom{u}{e}$
linear varieties arising as intersections of $e$ arbitrary
distinct choices $A_{i_1},\ldots,A_{i_e}$ of the hyperplanes.
Let $N\geq 1$ be an integer and regard $\pr n$ as a linear subvariety of $\pr {n+N}$.
The galaxy ${\mathcal G}={\mathcal G}(n,N,e,u,h)={\mathcal G}(n,N,e,u,h;S(n,e,u),{\mathcal H})$
consists of $S(n,e,u)$ and a choice of $h$ general points
${\mathcal H}=\{P_1,\ldots,P_h\}\subset\pr{n+N}$.
We refer to $S(n,e,u)$ as the galactic center, to $\pr n$ as the galactic ($n$-)plane, and to
${\mathcal H}$ as the galactic halo.

When $h=N$, this is an inclic ${\mathcal C}={\mathcal C}(n+N,r,0; L_0, L_1,\ldots,L_r, H_0)$ in which $L_0=P_h$,
$H_0$ is the linear span of $\pr n$ and the points $P_1,\ldots,P_{h-1}$, $r=\binom{u}{e}+h-1$
and the linear varieties $L_i$ are the points $P_j$, $j<h$ and the $e$-wise intersections
comprising the elements of $S(n,e,u)$.

\subsection{Hilbert functions, symbolic powers, Waldschmidt constants and resurgences}\label{HfWc}
Let $R=K[\pr n]=K[x_0,\ldots,x_n]$. For any linear subvariety $V\in\pr n$,
the ideal $I_V\subset R$ is the ideal generated by all forms vanishing on $V$.
Let $l_i$ and $h_j$ be non-negative integers.
An inclic scheme is a scheme of the form $X=\sum_{i\geq0}l_iL_i+\sum_{j>0}h_jH_j$,
by which we mean the scheme defined by the ideal
$I_X=(\cap_{i\geq0}I_{L_i}^{l_i})\bigcap(\cap_{j>0}I_{H_j}^{h_j})$.
We note that such ideals  are homogeneous and saturated.
Moreover, if $I=\sqrt{I_X}$, then $I=(\cap_{i\geq0}I_{L_i})\bigcap(\cap_{j>0}I_{H_j})$,
and for any $m\geq 1$, the symbolic power $I^{(m)}$ is $I^{(m)}=I_X$
in the case that $l_i=h_i=m$ for all $i$ and $j$.

For any homogeneous ideal $I\subseteq R$, the Hilbert
function of $I$ is the function $h(I)$ defined as
$h(I,t)=\dim_KI_t$, where $I_t$ is the $K$-vector space span of all
forms in $I$ of degree $t$. If $I_X\subsetneq R$ is the saturated
ideal defining a subscheme $X\subseteq \pr{n}$,
the Hilbert function of $X$ is the function
$h(X,t)=h(R,t)-h(I_X,t)=\binom{t+n}{n}-h(I_X,t)$. In all cases, we
adopt the understanding that Hilbert functions are $0$ when $t<0$.

An important value associated to any homogeneous ideal $(0)\neq I\subseteq R$ is $\alpha(I)$,
defined to be the least degree $t$ such that $h(I,t)\neq 0$.
We then define the {\em Waldschmidt constant} (introduced by Waldschmidt in \cite{W} in case $I$ is the ideal of
a finite set of points) to be
$$\gamma(I)=\lim_{m\to\infty}\frac{\alpha(I^{(m)})}{m}.$$
This limit exists but in general is hard to compute and not many specific values are known.

Another asymptotic measure related to $\gamma(I)$ which is also hard to compute
is the {\em resurgence} \cite{BH,BH2,GHV1}, defined for
any homogeneous ideal $(0)\neq I\subsetneq R=K[\pr n]$ as
$$\rho(I)=\sup\Big\{\frac{m}{r}: I^{(m)}\not\subseteq I^r\Big\}.$$
In general it is difficult to determine for which $m$ and $r$ we have
$I^{(m)}\subseteq I^r$. The interest of $\rho(I)$ is that it is
the largest real $c$ such that we always have $I^{(m)}\subseteq I^r$ for $m/r>c$,
but it is difficult to compute. It is not a priori even clear that it exists.
It is known  and easy to see that $1\leq \rho(I)$. Much deeper is the fact that
$I^{(m)}\subseteq I^r$ whenever $m/r\geq n$ \cite{ELS,HoHu} from which it follows that
$\rho(I)\leq n$ and hence $\rho(I)$ exists. This raised the issue of
getting better bounds. One of the main results for bounding and sometimes computing
$\rho(I)$ is that of \cite{BH} which says that $\frac{\alpha(I)}{\gamma(I)}\leq \rho(I)$,
and, if $I$ defines a 0-dimensional
subscheme of $\pr n$, that $\rho(I)\leq \frac{\reg(I)}{\gamma(I)}$,
where $\reg(I)$ is the Castelnuovo-Mumford regularity of $I$,
but these bounds depend on $\gamma(I)$ which has so far been computed in
relatively few cases.

\subsection{An application}\label{anapp}
Our work here builds on the results of \cite{FHL, FaL}
concerning finding Hilbert functions of fat points for points lying in a hyperplane,
but now we allow linear spaces of higher dimension, not all of which need to be contained in
a hyperplane. As an application of our results we compute galactic Waldschmidt constants.
To state our result, let ${\mathcal G}={\mathcal G}(n,N,e,u,N)$.
Let $G\subset \pr {n+N}$ be the reduced Galactic inclic scheme whose components are the
elements of ${\mathcal G}$;
i.e., $G$ is the reduced scheme theoretic union of
the $N$ points of ${\mathcal G}$ and the $\binom{u}{e}$ $e$-wise intersections of the
associated star configuration $S(n,e,u)$. Then we have:

\begin{thm}\label{galG}
Let $G$ be a reduced galactic inclic scheme as above.
\begin{itemize}
\item[(a)] We have $$\frac{2}{\gamma(I_G)}\leq \rho(I_G).$$
\item[(b)] If $K$ has characteristic 0, then
$$\gamma(I_G)=\frac{N(u-e)+u}{N(u-e)+e}.$$
\item[(c)] If $e=n$, then
$$\rho(I_G)\leq \frac{u-n+1}{\gamma(I_G)}.$$
\end{itemize}
\end{thm}

In the case that $h=N=1$, $e=n$ and $u=n+1$, we can, up to choice of coordinates,
regard $G$ as the coordinate vertices in $\pr{n+1}$. In this case, $u=n+1$ so the ideal $I_G$
can be chosen to be a monomial
ideal and we recover the known
facts that $\gamma(I_G)=\frac{n+2}{n+1}$ and $\rho(I_G)=\frac{2(n+1)}{n+2}$.
We note that the bound $\frac{2}{\gamma(I_G)}\leq \rho(I_G)$
is always better than the bound $1\leq \rho(I_G)$, and the bound
$\rho(I_G)\leq \frac{s-n+1}{\gamma(I_G)}$ is often better than the bound
$\rho(I_G)\leq n+N$ (such as if $u$ is not too big, say $u\leq 2n+N$).

The foundation for these results is our work on inclics. In section \ref{prelims}
we prove a lemma we will use later and we set up our technical notation.
In section \ref{mainthmsec} we prove our main foundational results
for inclics. In section \ref{galsec} we prove Theorem \ref{galG}.

\section{Preliminaries}\label{prelims}

Let $X\subset\pr n$ be a set of $c$ points regarded as a reduced subscheme.
It is well known that $\reg(I_X)=\tau+1$ where $\tau$
is the least degree $t$ such that the points impose independent conditions on
forms of degree $t$ (i.e., such that $h(I_X,t)=\binom{t+n}{n}-c$).

\begin{lem}\label{reglem}
Let $H\subset\pr n$ be a hyperplane and let $X\subset\pr n$ be a set of $c+1$
points regarded as a reduced subscheme,
with exactly $c$ of the points lying in $H$. Let $X'=X\cap H,$ then $\reg(I_{X'}) = \reg(I_X)$.
\end{lem}

\begin{proof} Choose coordinates such that $I_{H}=(x_0)$, where
$K[\pr n]=K[x_0,\ldots,x_n]$ and so $K[H]=K[x_1,\ldots,x_n]$.
Let $\tau'=\reg(I_{X'})-1$ and let $\tau=\reg(I_{X})-1$. Thus the points of $X'$
impose independent conditions on forms of degree $\tau'$ in $K[H]$, and
hence also in $K[x_0,\ldots,x_n]$.
Let $p$ be the point of $X$ not in $H$;
up to choice of coordinates
we can regard $p$ as being general, hence
it imposes an additional independent condition. Thus $\tau\leq \tau'$.

On the other hand, it follows from \cite[Corollary 3.3]{FrL} and from \cite[Proposition 2.1]{DG} that
$\tau'\leq\tau$, hence $\tau=\tau'$, so $\reg(I_{X'}) = \reg(I_X)$.
\end{proof}

Hereafter we study fat inclic schemes for some fixed hyperplane $H_0\subset \pr n$.
Clearly, we may choose coordinates such that $I_{H_0}=(x_0)$,
so $R'=K[\pr {n-1}]=K[H_0]=K[x_1,\ldots,x_n]$, and
$I_{L_0}=(x_{k+1},\ldots,x_n)\subset K[\pr n]=R$.
We fix such a choice of coordinates for the rest of this article.
We denote the linear forms defining $H_j$ for $j>0$ by $\eta_j$.
We also take $Y$ to be the fat subscheme $Y=l_1L_1+\cdots+l_rL_r$
of $\pr {n}$. In addition we define $Y'=Y\cap H_0$ and
$Y'_i=Y_i\cap H_0$ for $Y_i=l_1(i)L_1+\cdots+l_r(i)L_r$, where
$l_j(i)=\max(0,l_j-i)$. Thus $Y'_0=Y'$ and $I_{Y'_i}=I_{Y_i}\cap K[x_1,\dots,x_n]$.
Moreover, $I_{Y_i}=I_Y:(x_0^i)$.
We set $Z=l_0L_0$, $W=Y\cup Z$, $X=W\bigcup\cup_{j>0}h_jH_j$,
$L_0'=L_0\cap H_0$, $Z'=Z\cap H_0$, $W'=W\cap H_0$ and $X'=X\cap H_0$.

The following notation will be useful. If $J'\subseteq R'=K[x_1,\dots,x_n]$
is a homogeneous ideal, set $J'^{(k,t)}=J'\cap (I_{L_0}')^t$. Thus
$(J'^{(k,t)})_i=(J')_i\cap ((I_{L_0}')^t)_i$. Note that in the special case that
$k=0$ (i.e., that $L_0$ is the point $p$ defined in $\pr {n}$ by $(x_1,\ldots,x_n)$),
we have $(J'^{(0,t)})_i=J'_i$ for $i\geq t$
and $(J'^{(0,t)})_i=0$ for $i< t$; in short, if we know $J'$, then we
immediately know $J'^{(0,t)}$ for all $t$.

Note that $R$ has a bi-grading; i.e., the direct sum $R=\oplus_{ij}R_{ij}$
has the property that $R_{ij}R_{st}=R_{i+s,j+t}$,
where $R_i$ is the $K$-vector space span of the forms in
$R'=K[x_1,\ldots,x_n]$ of total degree $i$, and
$R_{ij}$ is the $K$-vector subspace $x_0^jR_i\subset R$.
We say an element $F\in R$ is bi-homogeneous if $F\in R_{ij}$
for some $i$ and $j$, and
we say an ideal $I\subseteq R$ is bi-homogeneous if
$I=\oplus_{ij}I_{ij}$, where $I_{ij}=I\cap R_{ij}$. As usual,
$I$ is bi-homogeneous if and only if $I$ has bi-homogeneous generators,
and intersections, sums and products of bi-homogeneous ideals are bi-homogeneous.

\section{Hilbert Functions}\label{mainthmsec}

We can now state and prove our main theorem.

\begin{thm}\label{HF}
Let $Y'$, $Y'_i$,  $Z$, $W$ and $X$ be as above, let $l'=\max(l_1,\dots,l_r)$. Then
$I_X=\eta_1^{h_1}\cdots\eta_s^{h_s}I_W$ and $h(I_X,t)=h(I_W,t-\sum_{j>0}h_j)$, where
$$I_W=\oplus_jx_0^j(I_{Y'_j})^{(k,l_0)}=
(\oplus_{0\leq j< l'}x_0^j(I_{Y'_j})^{(k,l_0)})\bigoplus\oplus_{j\geq l'} x_0^jI_{Z'}=
(\oplus_{0\leq j< l'}x_0^j(I_{Y'_j})^{(k,l_0)})\bigoplus x_0^{l'}I_Z,$$
and $h(I_W,t) =\sum_{j=0}^\lambda h((I_{Y'_j})^{(k,l_0)},t-j) +h(I_Z,t-l')$, where
$h(I_Z,t-l')=0$ if $t<l'+l_0$ and
$h(I_Z,t-l')= \binom{t-l'+n}{n}-\sum_{0\leq i<l_0}\binom{t-l'-i+k}{k}\binom{i+n-k-1}{n-k-1}$
for $t\geq l'+l_0$.
\end{thm}

\begin{proof} It is obvious that
$I_X=\eta_1^{h_1}\cdots\eta_s^{h_s}I_W$ and $h(I_X,t)=h(I_W,t-\sum_{j>0}h_j)$,
so now we consider $I_W$ and $h(I_W,t)$. To begin, note that the
ideals $I_{L_i}\subset R$ are bi-homogeneous (having bi-homogeneous
generators), so $I_Y$ and $I_W=I_Y\cap I_Z$ are bi-homogeneous, hence
$I_Y=\oplus_{ij}(I_Y)_{ij}$ and $I_W=\oplus_{ij}((I_Y)_{ij}\cap (I_Z)_{ij})$.
But $F\in (I_Y)_{ij}$ if and only if
$F=x_0^jG$ where $G\in (I_{Y'_j})_i$; i.e., $(I_Y)_{ij}=x_0^j(I_{Y'_j})_i$.
Thus $I_Y=\oplus_{ij}x_0^j(I_{Y'_j})_i$, and since
$(I_Z)_{ij}=x_0^j(I_{Z'})_i$, we have
\begin{equation}\tag{*}
I_W=\oplus_{ij}((x_0^j(I_{Y'_j})_i)\cap(I_Z)_{ij})=
\oplus_{ij}((x_0^j(I_{Y'_j}\cap I_{Z'})_i)=
\oplus_{ij}(x_0^j((I_{Y'_j})^{(k,l_0)})_i)=\oplus_jx_0^j(I_{Y'_j})^{(k,l_0)}.
\end{equation}
But for $j\geq l'$ we have $I_{Y'_j}=R'$ and hence $(I_{Y'_j})^{(k,l_0)}=I_{Z'}$, so
$$I_W=\oplus_jx_0^j(I_{Y'_j})^{(k,l_0)}=
(\oplus_{0\leq j< l'}x_0^j(I_{Y'_j})^{(k,l_0)})\bigoplus\oplus_{j\geq l'} x_0^jI_{Z'}=
(\oplus_{0\leq j< l'}x_0^j(I_{Y'_j})^{(k,l_0)})\bigoplus x_0^{l'}I_Z.$$
The fact that $h(I_W,t)=\sum_{j=0}^{l'-1} h((I_{Y'_j})^{(k,l_0)},t-j)+ h(I_Z,t-l')$
is now immediate, keeping in mind that the Hilbert function is computed with respect
to the singly graded structure of $R$; i.e., $(I_W)_t=\oplus_{i+j=t}(I_W)_{ij}$.
But the value of $h(I_Z,t-l')$ is known; the formula
given in the statement of the theorem comes from \cite[Lemma 2.1]{DHST}.
\end{proof}

The case with $k=0$ is particularly simple; in this case, if we know the Hilbert functions
of $Y'_j$ for all $j$, then we know the Hilbert functions of $W$ and hence $X$.

\begin{cor}\label{HFcor}
Under the hypotheses of Theorem \ref{HF}, let $\lambda=\min(l'-1,t-l_0).$ If we also have $k=0$, then
$$h(I_W,t)=\sum_{j=0}^\lambda h(I_{Y'_j},t-j)+ \sum_{j=l'}^{t-l_0} \binom{t-j+n-1}{n-1},$$
which is $h(I_W,t)=\sum_{j=0}^\lambda h(I_{Y'_j},t-j)$ for $t<l'+l_0$ and
$$h(I_W,t)=\sum_{j=0}^\lambda h(I_{Y'_j},t-j)+ \binom{t-l'+n}{n}-\binom{l_0+n-1}{n}$$
for $t\geq l'+l_0$.
\end{cor}

\begin{proof} This follows immediately from Theorem \ref{HF}, since
$(I_{Z'})_i=R'_i$ (so $h(I_{Z'},t-j)=\binom{t-j+n-1}{n-1}$ and
$((I_{Y'_j})^{(0,l_0)})_{t-j}=(I_{Y'_j})_{t-j}$ for $t-j\geq l_0$, that is, for $j\leq t-l_0$).
\end{proof}

\begin{remark}\label{HFrem}
Examples for which we would know the Hilbert functions of
$Y'_j\subset H_0$ for all $j$ can be constructed inductively.
For example, start with a flag of projective spaces $V_1\subset V_2\subset \cdots \subset V_n$,
each contained in the next as a linear subvariety, with $V_i\simeq \pr {i}$.
Let $U_1\subset V_1$ be any finite set of points $u_{11},\ldots, u_{1s}$.
Let $U_2$ consist of a point $u_{21}\in V_2\setminus V_1$ together with any lines
$u_{22},\ldots,u_{2s_2}\subset V_2$ not containing
$u_{21}$ or any component of $U_1$ (i.e., not containing $u_{1i}$ for any $i$).
Continue in this way, so
$U_i$ consists of a point $u_{i1}\in V_i\setminus V_{i-1}$ and a finite set of hyperplanes
$u_{ij}\subset V_i$ not containing $u_{i1}$ and not containing any of the
components of $U_j$ for $j<i$.
Then $U_1\cup\cdots\cup U_n$ defines an inclic and for any multiplicities
$m_{ij}$ we can inductively compute $h(I_X,t)$ for any $t$, for $X=\sum_{ij}m_{ij}u_{ij}$.
Indeed, define $X_1=\sum_jm_{1j}u_{1j}$, and then $X_2=X_1+\sum_jm_{2j}u_{2j}$,
and in general $X_k=X_{k-1}+\sum_jm_{kj}u_{kj}$. Since we know
$h((X_1)_i,t)$ for all $i$ and $t$, Theorem \ref{HF} gives us
$h((X_2)_i,t)$ for all $i$ and $t$, and similarly $h((X_k)_i,t)$ for each $k$
in turn for all $i$ and $t$, and hence eventually $h(X_n,t)$ for all $t$.

Our result also handles other constructions. For example, instead of starting with points
in $\pr 1$, we could start with a star configuration of points in $\pr 2$ (i.e., the points of
pair-wise intersection of a finite set of lines, no three of which meet at
any single point, see \cite{GHM}). Let $S$ be the scheme theoretic sum of the points of the star
and consider the scheme $iS$. The Hilbert function of $iS$ is known for all $i$
(\cite{CHT}), so we can proceed as above to construct an $X_n$, as long as
in this case we assign the same multiplicity to each point of $S$
(the Hilbert function is not always known if the multiplicities of the points of $S$
are allowed to vary).
\end{remark}

Given a closed subscheme $X\subset \pr n$ with corresponding ideal
$I_X$, define $\alpha(X)=\alpha(I_X)$ to be the least
degree $t$ such that there is a non-trivial form $F\in (I_X)_t$.

\begin{lem}\label{alfadecrlem}
With the  previous notation, there
is a least $j\geq0$ such that $\alpha((I_{Y'_j})^{(k,l_0)})= l_0$.
Let $l'=\max(l_1,\dots,l_r)$ and let $d$ be this least $j$.
If, moreover, $\operatorname{char}(K)=0$, then
$$0=\alpha(I_{Y'_{l'}})<\alpha(I_{Y'_{l'-1}})<\alpha(I_{Y'_{l'-2}})
<\ldots<\alpha(I_{Y'_0})$$
and
$$l_0=\alpha((I_{Y'_d})^{(k,l_0)})<\alpha((I_{Y'_{d-1}})^{(k,l_0)})<\alpha((I_{Y'_{d-2}})^{(k,l_0)})
<\ldots<\alpha((I_{Y'_0})^{(k,l_0)}).$$
\end{lem}

\begin{proof}
By definition, $(I_{Y'_j})^{(k,l_0)}\subset (x_{k+1},\ldots,x_n)^{l_0}$
so $\alpha((I_{Y'_j})^{(k,l_0)})\geq l_0$ for all $j$.
But for $j\geq l'$ we have $((I_{Y'_j})^{(k,l_0)})=(x_{k+1},\ldots,x_n)^{l_0}$
so $\alpha((I_{Y'_j})^{(k,l_0)})=l_0$.
Thus there is a least $j$ such that $\alpha((I_{Y'_j})^{(k,l_0)})= l_0$
so $d$ is defined.

Since ${Y'_{l'}}=\varnothing$, we have $I_{Y'_{l'}}=(1)$, so $\alpha(I_{Y'_{l'}})=0$.
Now assume $\operatorname{char}(K)=0$. Consider any
non-zero homogeneous element $F\in I_{Y'_j}$ for $j<l'$.
By Euler's identity, not all of the partials of $F$ are $0$. However, they all
belong to $I_{Y'_{j+1}}$ and the non-zero ones have degree $\deg(F)-1$. Therefore
$\alpha(I_{Y'_{j+1}})\leq\alpha(I_{Y'_j})-1$,
so we have
$$0=\alpha(I_{Y'_{l'}})<\alpha(I_{Y'_{d-1}})<\alpha(I_{Y'_{d-2}})
<\ldots<\alpha(I_{Y'_0})$$
as claimed.

The argument for the second claim is similar. Let $j<d$ and
consider any non-zero homogeneous element
$F\in (I_{Y'_j})^{(k,l_0)})\subseteq I_{Y'_j}$, so $\deg(F)>l_0$,
since also $F\in (I'_{L_0})^{l_0}$.
Again not all of the partials of $F$ are $0$ but they all
belong to $I_{Y'_{j+1}}$. Since $\deg(F)>l_0$, they all also belong to
$(x_{k+1},\dots,x_n)^{l_0}$ and hence to $(I_{Y'_{j+1}})^{(k,l_0)}$.
Therefore $\alpha((I_{Y'_{j+1}})^{(k,l_0)})\leq\alpha((I_{Y'_j})^{(k,l_0)})-1$,
so we have
$$l_0=\alpha((I_{Y'_d})^{(k,l_0)})<\alpha((I_{Y'_{d-1}})^{(k,l_0)})<\alpha((I_{Y'_{d-2}})^{(k,l_0)})
<\ldots<\alpha((I_{Y'_0})^{(k,l_0)}).$$
\end{proof}

\begin{cor}\label{alfa}
Let $L_0, L_1,\ldots,L_r, H_0, H_1,\ldots,H_s\subsetneq \pr {n}$ be an inclic, and let
$W=\sum_{i\geq 0}l_iL_i$ and
$X=\sum_{i\geq 0}l_iL_i+\sum_{j>0}h_jH_j$ for
non-negative integers $l_i$ and $h_i$.
Let $Y'_i$ be as above, and let $l'=\max(l_1,\ldots,l_r)$ and $h=h_1+\cdots+h_s$.
Then $\alpha(X)=h+\alpha(W)$ and
$\max(l',l_0)\leq \alpha(W)\leq l'+l_0$. Moreover, there
is a least $j\geq0$ such that $\alpha((I_{Y'_j})^{(k,l_0)})= l_0$. Taking this least $j$ to be $d$, we have
$$\alpha(W)\leq l_0+d,$$
with $\alpha(W)= l_0+d$ if $\operatorname{char}(K)=0$.
\end{cor}

\begin{proof}
Since $I_X=\eta_1^{h_1}\cdots\eta_s^{h_s}I_W$, where
$\eta_i$ is the linear form defining $H_i$, we see that
$\alpha(X)=h+\alpha(W)$. Since
$l_1,\ldots,l_r,l_0\leq \alpha(W)$, the lower bound
$\max(l',l_0)\leq \alpha(W)$ holds. Since
$x_0^{l'}x_n^{l_0}\in I_W$,  the upper bound
$\alpha(W)\leq h+l'+l_0$ also holds.

To get more precise information, note by $(^*)$ in the proof of Theorem \ref{HF} that
$$\alpha(W)=\alpha(I_W)=\min_j\alpha(x_0^j(I_{Y'_j})^{(k,l_0)})=\min_j(j+\alpha((I_{Y'_j})^{(k,l_0)})).$$

By definition, $(I_{Y'_j})^{(k,l_0)}\subset (x_{k+1},\ldots,x_n)^{l_0}$
so $\alpha((I_{Y'_j})^{(k,l_0)})\geq l_0$ for all $j$.
But for $j\geq l'$ we have $((I_{Y'_j})^{(k,l_0)})=(x_{k+1},\ldots,x_n)^{l_0}$
so $\alpha((I_{Y'_j})^{(k,l_0)})=l_0$ and hence $\alpha(I_W)\leq l'+l_0$.
Thus there is a least $j$ such that $\alpha((I_{Y'_j})^{(k,l_0)})= l_0$
so $d$ is defined. Thus
we have $\alpha(I_W)\leq d+l_0$, and in addition we have
$d+l_0\leq j+ \alpha((I_{Y'_j})^{(k,l_0)})$
for all $j\geq d$.

Now assume $\operatorname{char}(K)=0$. By Lemma \ref{alfadecrlem} we have
$$l_0+d\leq\alpha((I_{Y'_{d-1}})^{(k,l_0)})+(d-1)\leq\alpha((I_{Y'_{d-2}})^{(k,l_0)})+(d-2)\leq\ldots
\leq\alpha((I_{Y'_0})^{(k,l_0)})+(d-d),$$
and hence $\alpha(I_{{Y'_j}^{(k,l_0)}})+j\geq l_0+d$ for all $j<d$, and therefore
$\alpha(W)=l_0+d$, as claimed.
\end{proof}

\section{Galaxies}\label{galsec}

In this section we prove Theorem \ref{galG}. In order to compute $\gamma(I_X)$, we will
determine $\alpha((I_X)^{(j)})$ for an unbounded sequence of values of $j$.
Our inductive procedure requires information about star configurations as a starting point.
The following result is from \cite{BH}.

\begin{thm}\label{BHthm}
Let $1\leq e\leq n<u$ be integers.
Let $A\subset\pr n$ be the reduced scheme theoretic union of the linear varieties
comprising the star configuration $S(n,e,u)$. Then $\alpha(reA)=ru$ for any integer $r\geq 1$,
$\alpha(A)=u-e+1$ and, if $e=n$, $\reg(I_A)=u-n+1$.
\end{thm}

\begin{proof}[{Proof of Theorem \ref{galG}}]
Let ${\mathcal G}={\mathcal G}(n,N,e,u,N;S(n,e,u),{\mathcal H})$ and
let the points of the galactic halo ${\mathcal H}$ be $P_1,\ldots,P_N$.
Let $G_0=S(n,e,u)=A\subset\pr n$, $G_1=A+P_1\subset\pr {n+1}, \ldots,
G=G_N=A+P_1+\cdots+P_N\subset\pr {n+N}$.

The bounds on $\rho(I_{G})$
come from $\alpha(I_G)/\gamma(I_G)\leq \rho(I_G)$ and,
when $e=n$, $\rho(I_G)\leq\reg(I_G)/\gamma(I_G)$ \cite{BH}.
Since $G$ spans $\pr{n+N}$, we see $1<\alpha(I_G)$, but
$G$ is contained in the span of $\pr n$ and $N$ points, each of which
is contained in a hyperplane in $\pr{n+N}$, so $\alpha(I_G)\leq 2$;
thus $\alpha(I_G)=2$. And by Lemma \ref{reglem} with $e=n$, $\reg(I_G)=\reg(I_{G_0})$,
but $\reg(I_{G_0})=u-n+1$ by Theorem \ref{BHthm}.

Now define the following sequence: $a_0=re$, $a_1=ru$, and for $i\geq0$, let
$a_{i+2}=2a_{i+1}-a_i$. It's easy to check that
$a_i=iru-(i-1)re$.
In what comes below, for each $i$ we regard $I_{G_i}$ as an ideal in $K[\pr{n+i}]$.
We begin by noting that $a_1=\alpha(I_{a_0G_0})$.
We will show by induction that $a_{i+1}=\alpha(I_{a_iG_i})$, and hence that
$\alpha(I_{a_NG_N})=(N+1)ru-Nre$, so
$\gamma(I_{G_N})=\lim_{r\to\infty}((N+1)ru-Nre)/(Nru-(N-1)re)=
((N+1)u-Ne)/(Nu-(N-1)e)$, as claimed.

To show that $a_{i+1}=\alpha(I_{a_iG_i})$ we will apply Corollary \ref{alfa}.
The $W$ of Corollary \ref{alfa} is $G_i$; $L_0=P_i$ and the $L_j$, $j>0$ are the components
of $A$ and the points $P_1,\ldots,P_{i-1}$; $H_0$ is the linear span of
$\pr n$ and $P_1,\ldots,P_{i-1}$; $\pr n$ there is $\pr {n+i}$ here; and $l'=l_j=a_i$ for all $j$.
Moreover, $k$ in the corollary is 0, since $L_0$ is a point.
(Here there are no $H_j$ for $j>0$, and $l_0=a_i$.)
The result of the corollary is that $\alpha(a_iG_i)=a_i+d$, where $d$ is the least $j$
such that what is there called $Y'_j$ has $\alpha((I_{Y'_j})^{(0,l_0)}=l_0$.
But $Y'_j=(a_i-j)G_{i-1}$ (as long as $a_i-j\geq 0$), and $(I_{Y'_j})^{(0,l_0)}$
is just the truncation of $I_{Y'_j}$ at degree $l_0$. Thus the least $j$ such
that $\alpha((I_{Y'_j})^{(0,l_0)})=l_0$ is the least $j$ such that
$\alpha((a_i-j)G_{i-1})\leq a_i$.
But $\alpha((a_i-j)G_{i-1})=a_i$ for $j=a_i-a_{i-1}$
by induction, and $\alpha((a_{i-1})G_{i-1})<\alpha((a_{i}-j)G_{i-1})$ for $j<a_i-a_{i-1}$
by Lemma \ref{alfadecrlem}, so $a_i-a_{i-1}$ is the least $j$. Thus
$\alpha(I_{a_iG_i})=a_i+(a_i-a_{i-1})=a_{i+1}$, as claimed.
\end{proof}


\begin{thebibliography}{MMM}

\bibitem[BH]{BH} C.~Bocci and B.~Harbourne,
{\it Comparing powers and symbolic power of ideals},
J.~Algebraic Geom. {\bf 19} (2010), 399--417.

\bibitem[BH2]{BH2}
C.~Bocci and B.~Harbourne,
{\it The resurgence of ideals of points and the containment problem},
Proc.~Amer.~Math.~Soc.  {\bf 138}  (2010), 1175--1190.

\bibitem[Ch]{Ch}  G.~V.~Chudnovsky. {\it Singular points on complex
hypersurfaces and multidimensional
Schwarz Lemma}, S\'eminaire de Th\'eorie des Nombres, Paris 1979--80,
S\'eminaire Delange-Pisot-Poitou, Progress in Math vol. 12, M-J Bertin, editor,
Birkh\"auser, Boston-Basel-Stutgart (1981).

\bibitem[CHT]{CHT} S.~Cooper, B.~Harbourne  and Z.~Teitler.
{\it Combinatorial bounds on Hilbert functions of fat points in projective space},
J.~Pure~Appl.~Algebra, Volume 215, Issue 9, September 2011,
Pages 2165--2179, arXiv:0912.1915.

\bibitem[D]{D} H.~Derksen. {\it Hilbert series of subspace arrangements},
J. Pure Appl. Algebra 209 (2007), no.1, 91-98.

\bibitem[DG]{DG} E.~Davis, A.V.~Geramita {\it The Hilbert function of a special class
of $1$-dimensional Cohen-Macaulay graded algebras}, The Cur5ves Seminar at Queen's, Vol. X,
Queen's Papers in pure and Applied Mathematics, Vol.67 (1984), 1H-29H.

\bibitem[DS]{DS} H.~Derksen and J.~Sidman. {\it A sharp bound for the Castelnuovo-Mumford regularity
of subspace arrangements}, Adv. Math. 172 (2002), no. 2, 151--157.

\bibitem[DHST]{DHST} M.~Dumnicki, B.~Harbourne, T.~Szemberg and H.~Tutaj-Gasi\'nska.
{\em Linear subspaces, symbolic powers and Nagata type conjectures}, preprint 2012, 20 pp., arXiv:1207.1159.

\bibitem[ELS]{ELS} L. Ein, R. Lazarsfeld, and K.E. Smith.
{\it Uniform behavior of symbolic powers of ideals}, Invent. Math.,
144 (2001), 241--252.

\bibitem[EV]{EV}
Hd.\ Esnault and E.\ Viehweg.
{\it Sur une minoration du degr\'e d'hypersurfaces
s'annulant en certains points}, Math. Ann. 263 (1983), no. 1, 75--86.

\bibitem[Fa]{Fa} G.~Fatabbi.
{\it On the resolution of ideals of fat points}, J.~Algebra 242 (2001), 92--108.

\bibitem[FHL]{FHL} G.~Fatabbi, B.~Harbourne and A.~Lorenzini.
{\it Resolutions of ideals of fat points with support in a hyperplane},
Proc.~Amer.~Math.~Soc. 134 (2006) 3475--3483.

\bibitem[FaL]{FaL} G.~Fatabbi and A.~Lorenzini. {\it On the graded
resolution of ideals of a few general fat points of $\pr n$},
J.~Pure~Appl.~Algebra 198 (2005) 123--150.

\bibitem[FrL]{FrL}  S.~Franceschini and A.~Lorenzini.
{\it Fat points of $P^n$ whose support is contained in a linear
proper subspace}, J.~Pure~Appl.~Algebra 160 (2001), no. 2--3,
169--182.

\bibitem[Fr]{Fr} C.~A.~Francisco. {\it Resolution of small sets of fat points},
J.~Pure~Appl.~Algebra 203 (2005), 220--236

\bibitem[FMN]{FMN} C.~A.~Francisco, J.C.~Migliore and U.~Nagel.
{\it On the componentwise linearity and the minimal free
resolution of a tetrahedral curve}, J.~Algebra 299 (2006),
535--569

\bibitem[GHM]{GHM} A.~Geramita, B.~Harbourne and J.~Migliore.
{\em Star configurations in $\pr {n}$},
J.~Algebra 376 (2013) 279--299 (arXiv:1203.5685).

\bibitem[GMS]{GMS} A.~V.~Geramita, J.~Migliore and L.~Sabourin.
{\em On the first infinitesimal neighborhood of a linear configuration of
points in $\mathbb P^2$},
J.~Algebra 298 (2006),  563--611.

\bibitem[GO]{GO} A.~V.~Geramita and F.~Orecchia. {\it On the Cohen-Macaulay type of $s$ lines in
${\bf A}^{n+1}$}, J.~Algebra  70 (1981), 116--140.

\bibitem[GHV1]{GHV1} E.~Guardo, B.~Harbourne and A.~Van Tuyl.
{\em Asymptotic resurgences for ideals of positive dimensional
subschemes of projective space}, preprint 2012, 12pp., arXiv:1202.4370.

\bibitem[GHV2]{GHV2} E.~Guardo, B.~Harbourne and A.~Van Tuyl.
{\em Fat lines in $\pr {3}$: powers versus symbolic powers},
preprint 2012, 10 pp., arXiv:1208.5221.

\bibitem[HaHu]{HaHu} B.~Harbourne and C.~Huneke.
{\it Are symbolic powers highly evolved?},
To appear J. Ramanujan Math. Soc. {\tt arXiv:1103.5809v1}

\bibitem[HH]{HH} R.~Hartshorne and A.~Hirschowitz. {\it Droites en position g\'en\'erale dans
l'espace projectif}, volume 961 of Lecture Notes in Mathematics (1982), 169--188.

\bibitem[HS]{HS} A.~Hirschowitz and C.~Simpson. {\it Minimal resolution of the ideal of a general arrangement
of a large number of points in $\pr n$}, Inv. Math. 126 (1996), 467--503.

\bibitem[HoHu]{HoHu}
M.\ Hochster and C.\ Huneke. {\it Comparison of symbolic and
ordinary powers of ideals}, Invent. Math. {\bf 147}
(2002), no.~2, 349--369.

\bibitem[La]{La}
  R.\ Lazarsfeld.
 {\it Positivity in Algebraic Geometry
  I.-II}. Ergebnisse der Mathematik und ihrer Grenzgebiete, Vols.
  48-49, Springer Verlag, Berlin, 2004.

\bibitem[L]{L} A.~Lorenzini. {\it The minimal resolution conjecture},
J.~Algebra  156 (1993), 5--35.

\bibitem[M]{M} S.~Mayes. {\it The asymptotic behaviour of symbolic generic initial
systems of points in general position}, preprint 2012, 25 pp., arXiv:1210.1622.

\bibitem[V]{valla} G.~Valla. {\em Betti numbers of some monomial
ideals}, Proc.~Amer.~Math.~Soc., {\textbf 133} (2005), 57-63.

\bibitem[W]{W} M.~Waldschmidt.
{\it Propri\'et\'es arithm\'etiques de fonctions de plusieurs variables II},
In S\'eminaire P.~Lelong (Analyse), 1975--76, Lecture Notes Math. 578,
Springer-Verlag, 1977, 108--135.

\bibitem[W2]{W2} M.~Waldschmidt.
{\it Nombres transcendants et groupes alg\'ebriques},
Ast\'erisque 69/70, Soci\'ete Math\'ematiqu\'e de France, 1979.

\end{thebibliography}
\end{document}